\documentclass[10pt]{amsart}

\usepackage{amssymb}
\usepackage{amsmath}
\usepackage{epsfig}
\usepackage{color}
\usepackage{graphics}

\newcommand{\RR}{{\mathbb R}}

\newcommand{\NN}{{\mathbb N}}
\newcommand{\ZZ}{\mathbb Z}

\def\A{{\mathcal A}}

\def\M{{\mathcal M}}

\def\T{{\mathcal T}}

\def\supp{{\rm supp}}

\def\diam{{\rm diam}}
\def\vol{{\rm vol}}
\def\freq{{\rm freq}}
\def\dist{{\rm dist}}
\numberwithin{equation}{section}

\newtheorem{theo}{Theorem}
\newtheorem{prop}[theo]{Proposition}

\newtheorem{lemma}[theo]{Lemma}

\newtheorem{remark}[theo]{Remark}

\begin{document}

\parindent = 0cm

\title
{Self-similar tiling systems, topological factors and stretching factors}

\author{Mar\'{\i}a Isabel Cortez}
\address{Departamento de Matem\'atica y CC. de la Universidad de Santiago de Chile,
Av. Libertador Bernardo O'Higgins 3363.}
\email{mcortez@usach.cl }

\author{Fabien Durand}
\address{Laboratoire Ami\'enois
de Math\'ematiques Fondamentales et
Appliqu\'ees, CNRS-UMR 6140, Universit\'{e} de Picardie
Jules Verne, 33 rue Saint Leu, 80039 Amiens Cedex, France.}
\email{fabien.durand@u-picardie.fr}


\begin{abstract}
In this paper we prove 
that if two self-similar tiling systems, with respective  stretching factors $\lambda_1$ 
and $\lambda_2$, have a common factor which is a non periodic tiling system, then $\lambda_1$ 
and $\lambda_2$ are multiplicatively dependent.
\end{abstract}

\date{2007, July 10th}
\maketitle \markboth{Maria Isabel Cortez, Fabien Durand}{Self-similar tiling systems, topological factors and expansion constants.}

\section{Introduction}
 
 Given a non periodic self-similar tiling $\T$ generated by some similarity $S_1$ with stretching factor $\lambda_1$,
it is rather natural to ask if we could generate $\T$ using another similarity with a different 
stretching factor $\lambda_2$.
This is of course possible taking a power of the similarity $S_1$, where
$\lambda_2$ is in this  case a power of $\lambda_1$. Holton, Radin and Sadun show in \cite{HRS}
that  the stretching factor of any other similarity 
which generates $\T$   is equal to a rational power of $\lambda_1$. More precisely, they prove that the stretching factors of conjugate 
tiling systems which are the orbit closure under  Euclidean motions  of some self similar tilings are multiplicatively dependent.
In this paper we look at tiling systems  which are the orbit closure
under translations of some self similar tilings, in order to give a necessary condition to have non periodic common factors.
The result we present in this paper is the following:

\begin{theo}
\label{mainresult} Let $S_1(\T_1)=\T_1$ and $S_2(\T_2)=\T_2$ be two
self-similar tilings satisfying the Finite Pattern Condition, where $S_1$ and $S_2$ are primitive substitutions.
Let $\lambda_1$ and $\lambda_2$ be the
Perron eigenvalues of the substitution matrices associated to $S_1$
and $S_2$ respectively. If there exist a   non periodic
tiling $\T$ %satisfying the FPC 
and factors maps $\pi_i:\Omega_{\T_i}\to \Omega_{\T}$,
for $i\in \{1,2\}$, then $\lambda_1$ and $\lambda_2$ are
multiplicatively dependent.
\end{theo}

The problem we are interested in has been considered a long time ago by A. Cobham in \cite{Co1} 
and \cite{Co2} for fixed points of substitutions of constant length. 
He showed that if $p,q>1$  are two multiplicatively independent integers  then a sequence $x$ on a finite alphabet is both  $p$-substitutive and $q$-substitutive if and only if $x$ is ultimately periodic,
where $p$-substitutive means that $x$ is the image by a letter to letter morphism of a fixed point of a substitution
of constant length $p$.
This theorem was the starting point of a lot of work in many different directions such as~:
numeration systems for $\NN$, substitutive sequences and subshifts, automata theory and logic
(for more details see \cite{Be,BH1,BH2,BHMV,Du-2,Du-1,Du0,Ei,Fab,Fag,Ha1,Ha2,MV}). Later, in \cite{Se} A. Semenov proved a ``multidimensional'' Cobham type theorem, that is to say
a Cobham theorem for recognizable subsets of $\NN^d$.
This result can be stated in terms of self similar tilings,
and in the case these tilings are repetitive, our result is a generalization of Semenov Theorem.

\medskip

This paper is organized as follows: in Section \ref{definitions} we give some basic definitions relevant for the study of tiling systems and substitution tiling systems. In Section \ref{frequencies} we study the frequencies of the patches in   self-similar tilings   and in  their factors. First we prove that the frequencies of the patches in a self-similar tiling $\T$ are included in a finite union of geometric progressions of rate $\lambda$, where $\lambda$ is the  stretching factor of $\T$ (In \cite{HZ} the authors remarked this fact for minimal substitution subshifts). 
%Let us say some words about the proof.
%Let $\T_1$, $\T_2$ and $\T$ be as in Theorem \ref{mainresult}.
%The proof works as follows.
%From a result of B. Solomyak we know these systems are uniquely ergodic.
Next, we prove that the frequencies of the patches in a tiling $\T$, which is a factor of two self-similar tiling systems  with stretching factors $\lambda_1$ and $\lambda_2$ respectively,
%Then, we prove that the set of measures of some patches of $\T_i$ is included in a finite union of geometric progressions of rate $\lambda_i$.
%(In \cite{HZ} the authors remarked this fact for minimal substitution subshifts).
%Next, we consider the set of the measures of some patches in $\T$ and we show it 
are included in the intersection of two  
finite unions of geometric progressions, one of rate $\lambda_1$ and the other of rate $\lambda_2$. 
The proof of this result would be easier if  the factor maps were given by a kind of ``sliding block code''
%if the tiling spaces were Cantor spaces 
(as it can be the case for subshifts), because in this case the preimage of a patch would be a finite collection of patches.
Nevertheless, this is no longer the case 
for the tiling systems we consider here (examples of factor maps, and even conjugacies, that are not given by a ``sliding block codes'' are given in \cite{Pe} and \cite{RS}), 
%The difficulty was that, for factor maps between subshifts, the preimages of words are a finite collection of words, which analogous property is not %true for tiling systems. 
but we overcome this problem selecting carefully some patches in the preimages we considered. Finally, in Section 4 we  
deduce the main Theorem.
%Finally, in Section 4 we deduce the main Theorem. 

%%%%%%%%%%%%%%%%%%%%%%%%%%%%%%%%%%%%
\section{Definitions and background}
%%%%%%%%%%%%%%%%%%%%%%%%%%%%%%%%%%%%
\label{definitions} In this section we give the classical
definitions concerning tilings. For more details we refer to
\cite{So1}. A {\em tiling} of $\RR^d$ is a countable collection
$\T=\{t_i:i\geq 0\}$ of closed subsets of $\RR^d$ (which are known
as {\em tiles}) whose union is the whole space and their interiors
are pairwise disjoint. We assume that the tiles are homeomorphic to
closed balls and that they belong, up to translations, to a finite
collection of closed subsets of $\RR^d$ whose elements are called
{\em prototiles}. We say that two tiles are {\em equivalent} if they
are equal up to translations. It is often  useful to consider every
prototile as a closed set endowed with a label. In this case, two
tiles are equivalent if, in addition, their labels coincide.

The translation of the tiling $\T$ by a vector $v\in \RR^d$ is the
tiling $\T+v$ obtained after translating every tile of $\T$ by $-v$.
The tiling $\T$ is said to be {\em aperiodic} (or {\em non
periodic}) if $\T+v=\T$ implies $v=0$.

The {\em support} of a tile $t_i$, denoted by $\supp(t_i)$, is the
closed set that defines $t_i$. For every subset $A$ of $\RR^d$ we
define, as usual, $\T\cap A$ to be the set $\{t_i\cap A: i\geq 0\}$.
A {\em patch} $P$ is a finite collection of tiles. The support of a
patch $P$, denoted by $\supp(P)$, is the union of the supports of
the tiles in $P$. The {\em diameter} of a patch $P$ is the diameter
of its support, we call it $\diam(P)$. We define $P+v$ as we defined
$\T +v$. 

The tiling $\T$ satisfies the {\em finite pattern condition} FPC (or
equivalently, we say that it is {\em locally finite}) if for any
$r>0$, there are up to translation, only finitely many patches with
diameter smaller than $r$. This condition is automatically satisfied
in the case  of a tiling whose tiles are polyhedra that meet
face-to-face. A tiling $\T$ is {\em repetitive} if for any patch $P$
in $\T$ there exists $r>0$, such that  for every open ball $B_r(v)$
the collection $\T\cap B_r(v)$ contains a patch $P'$ equivalent to
$P$ (when it is clear from the context we will say that $P$
"appears" in $B_r(v)$). The non periodic repetitive tilings that
satisfy FPC are called {\em perfect} tilings.

%%%%%%%%%%%%%%%%%%%%%%%%%%%%
\subsection{Tiling systems}
%%%%%%%%%%%%%%%%%%%%%%%%%%%%
Let $\A$ be a finite collection of prototiles.  We denote by $T(\A)$
({\em full tiling space}) the space  of all the tilings of $\RR^d$
whose tiles are equivalent to some element in $\A$. We always
suppose that  $T(\A)$ is non empty.  The group $\RR^d$ acts on
$T(\A)$ by translations:
$$
(v,\T)\to \T+v  \mbox{ for } v\in \RR^d \mbox{ and } \T\in T(\A).$$
Furthermore, this action is continuous with the topology induced by
the following distance: take $\T$, $\T'$ in $T(\A)$, and define $A$
the set of $\varepsilon \in (0,1)$ such that there exist $v$ and
$v'$ in $B_{\varepsilon}(0)$  with  $$(\T+v)\cap
B_{1/\varepsilon}(0)=(\T'+v')\cap B_{1/\varepsilon}(0),$$ we set
$$
d(\T,\T')=\left\{\begin{array}{cc}
                  \inf A & \mbox{ if } A\neq \emptyset \\
                 1 & \mbox{ if } A=\emptyset.\\
               \end{array}\right.
$$
Roughly speaking, two tilings are close if they have the same
pattern in a large neighborhood of the origin, up to a small
translation. A {\em tiling system} is a pair $(\Omega, \RR^d)$ such
that $\Omega$ is a translation invariant closed subset of some full
tiling space. The orbit closure of the tiling $\T$ in $T(\A)$ is the
set $\Omega_{\T}=\overline{\{\T+v: v\in \RR^d\}}$. When $\T$
satisfies the FPC,  $\Omega_\T$ is compact (see \cite{Ru}). If $\T$
is repetitive then all the orbits are dense in $\Omega_{\T}$. In this
case the tiling system $(\Omega_{\T},\RR^d)$ is said to be {\em
minimal}.

\medskip

A {\em factor map}   between two tiling systems $(\Omega_{1},\RR^d)$
and $(\Omega_2,\RR^d)$ is a continuous map $\pi:\Omega_1\to
\Omega_2$ such that $\pi(\T+v)=\pi(T)+v$, for every $\T\in \Omega_1$
and $v\in \RR^d$.

In symbolic dynamics it is well-known that topological factor maps
between subshifts are always given by sliding-block-codes.  There
are examples  which show that this result can not be extended to
tiling systems (\cite{Pe}, \cite{RS}). The following Lemma shows
that factor maps between tiling systems are not far to be
sliding-block-codes. A similar result  can be found in \cite{HRS}.

\begin{lemma}
\label{semi-sliding-block-code} Let $\T_1$ and $\T_2$ be two tilings.
%verifying the FPC. 
Suppose $\T_1$ verifies the FPC and  $\pi: \Omega_{\T_1}\to \Omega_{\T_2}$ is
a factor map. Then, there exists a constant $s_0>0$ such that
to every   $\varepsilon>0$ it is possible to associate $R_{\varepsilon}>0$ 
satisfying the following: Let $R\geq R_{\varepsilon}$. If  $\T$ and $\T'$  in $\Omega_{\T_1}$ verify $$\T\cap
B_{R+s_0}(0)=\T'\cap B_{R+s_0}(0),$$  then
$$(\pi(\T)+v)\cap B_{R}(0)=\pi(\T')\cap B_{R}(0)$$ for some $v\in
B_{\varepsilon}(0)$. 
\end{lemma}

\begin{proof}
The tiling $\T_2$ also satisfies the FPC because $\Omega_{\T_2}$ is compact. 
Since the tilings in $\Omega_{\T_2}$ have a finite number of tiles,
up to translations, there exists $\delta_0'>0$ such that if $y_1\neq
y_2\in \RR^d$ satisfy $(\T+y_1)\cap B_R(0)=(\T+y_2)\cap B_R(0)$ for
some $\T\in \Omega_{\T_2}$ and some  $R> \max\{\diam(p): p \mbox{
prototile in }   \T\}$, then $\|y_1-y_2\|\geq \delta_0'$ (for the
details see \cite{So1}).

\medskip

Let $0<\delta_0<\frac{\delta_0'}{2}$. Since $\pi$ is uniformly
continuous, there exists $s_0>1$  such that if $\T$ and  $\T'$  in
$\Omega_{\T_1}$ verify $\T\cap B_{s_0}(0)=\T'\cap B_{s_0}(0)$ then
  $$(\pi(\T)+v)\cap B_{\frac{1}{\delta_0}}(0)=\pi(\T')\cap
B_{\frac{1}{\delta_0}}(0),$$  for some  $v\in B_{\delta_0}(0)$.

\medskip

Let $0<\varepsilon< \delta_0$. By uniform continuity of $\pi$ 
there exists $0<\delta<\frac{1}{s_0}$ such that if  $\T$ and  $\T'$  in
$\Omega_{\T_1}$ verify $\T\cap B_{\frac{1}{\delta}}(0)=\T'\cap B_{\frac{1}{\delta}}(0)$ then
\begin{equation} 
\label{continuidad} 
(\pi(\T)+v)\cap B_{\frac{1}{\varepsilon}}(0)=\pi(\T')\cap
B_{\frac{1}{\varepsilon}}(0),
\end{equation}
for some  $v\in B_{\varepsilon}(0)$.

Now fix $R\geq R_{\varepsilon}=\frac{1}{\delta} -s_0$ and  $\T$ and $\T'$ two tilings  in $\Omega_{\T_1}$  verifying
\begin{equation}
\label{motivo2} \T\cap B_{R+s_0}(0)=\T'\cap B_{R+s_0}(0).
\end{equation}
Then, on one hand,  the tilings $\T$  and, $\T'$ satisfy (\ref{continuidad}), and on the other hand, 
we obtain that $(\T+a)\cap B_{s_0}(0)=(\T'+a)\cap
B_{s_0}(0)$ for every $a$ in $B_{R}(0)$. From the choice of $s_0$, this implies that
\begin{equation}
\label{motivo1}
(\pi(\T)+a+t_a)\cap
B_{\frac{1}{\delta_0}}(0)=(\pi(\T')+a)\cap
B_{\frac{1}{\delta_0}}(0),
\end{equation}
for some $t_a\in B_{\delta_0}(0)$. 

Since $\delta_0 > \varepsilon$, from (\ref{continuidad}) we get
\begin{equation}
\label{motivo3}
(\pi(\T)+v)\cap
B_{\frac{1}{\delta_0}}(0)=\pi(\T')\cap B_{\frac{1}{\delta_0}}(0).
\end{equation}
We will show that $t_a=v$ for every $a$ in $B_{R}(0)$. This
property together with (\ref{motivo1}) and (\ref{motivo3})   imply that
$$(\pi(\T)+v)\cap B_{R}(0)=\pi(\T')\cap
B_{R}(0).$$
For $a=0$, from (\ref{motivo1}) and (\ref{motivo3}) we have that  $t_0=v$ or $\|v-t_0\|\geq \delta_0'$.
Since $\|t_0-v\|\leq \delta_0+\varepsilon <2\delta_0<\delta_0'$, we conclude $t_0=v$.\\
For $a\in B_{R}(0)$, consider $s>0$ such that for
every $a'\in B_s(a)$ the patch
$$
P=((\pi(\T')+a)\cap B_{\frac{1}{\delta_0}}(0))\cap
((\pi(\T')+a+(a'-a))\cap B_{\frac{1}{\delta_0}}(0),
$$
contains a tile.\\
From (\ref{motivo1}) we get  $\pi(\T)+a+t_a+(a-a')\cap \supp(P)=P$.
Replacing $a$ by $a'$ in (\ref{motivo1}), we obtain
$\pi(\T)+a+t_a'+(a'-a)\cap \supp(P)=P$. This implies the norm of
$t_a-t_a'$  is equal to $0$ or greater than $\delta_0'$. Since
$\|t_a-t_a'\|\leq 2\delta_0<\delta_0'$,  we get $t_a=t_a'$. Thus we
conclude that the function that associates $t_a$ to $a$ is constant,
which implies that $t_a=t_0=v$ for every $a$ in
$B_{R}(0)$.
\end{proof}

\subsection{Linearly recurrent tilings.}

A tiling $\T$ is {\it linearly recurrent} (or strongly repetitive,
or linearly repetitive) if there exists a constant $L>0$ such that
for every patch $P$ in $\T$, any ball of radius $L\diam(P)$ contains
a translate of $P$. Every tiling in the orbit closure of a linearly
recurrent tiling is linearly recurrent with the same constant. When
$\T$ is linearly recurrent, we call $(\Omega_{\T},\RR^d)$ a {\it
linearly recurrent} tiling system.

\begin{lemma}
Let $\T_1$ and $\T_2$ be two tilings
verifying the FPC. If $\pi: \Omega_{\T_1}\to \Omega_{\T_2}$ is a
factor map and $\T_1$ is linearly repetitive, then $(\Omega_{\T_2},\RR^d)$ is linearly recurrent.
\end{lemma}

\begin{proof}
Let   $\T\in \Omega_{\T_1}$. Consider  $\varepsilon>0$ and $R>0$ the positive number of Lemma \ref{semi-sliding-block-code}
associated to $\varepsilon$.  Since $\T$ is linearly repetitive with some constant $L$,  for any $y\in \RR^d$
there exists $v\in B_{L(R+s_0)}(y)$ such that $B_{R+s_0}(v)\subseteq  B_{L(R+s_0)}(y)$ and
$(\T+v)\cap B_{R+s_0}(0)=\T\cap B_{R+s_0}(0)$. From Lemma \ref{semi-sliding-block-code}, there exists
$t\in B_{\varepsilon}(0)$ such that $(\pi(\T)+v+t)\cap B_{R}(0)=\pi(\T)\cap B_R(0)$. This implies that any ball of radius $L(R+s_0)+2\varepsilon$ in $\pi(\T)$ contains a copy of $\pi(\T)\cap B_R(0)$. Since $Ls_0+2\varepsilon$ is smaller than some constant, it follows that $\pi(\T)$ is linearly recurrent.
\end{proof}

%%%%%%%%%%%%%%%%%%%%%%%%%%%%%%%%%%%%%%%%%%%
\subsection{Substitution tiling systems. }
%%%%%%%%%%%%%%%%%%%%%%%%%%%%%%%%%%%%%%%%%%%

%%%%%%%%%%%%%%%%%%%%%%%%%%%%%%%%%%%%%%%%%%%
%\subsubsection{Expansive maps.}
%%%%%%%%%%%%%%%%%%%%%%%%%%%%%%%%%%%%%%%%%%%
 Let $M$ be a linear map on $\RR^d$. It is called {\em expansive} if there exists
 $\lambda>1$ such that
 $$
 \|Mv\|\geq \lambda\|v\|, \mbox{ for all } v\in \RR^d.$$
The map $M$ is a {\em similarity} if $\|Mv\|=\lambda\|v\|$
for all $v\in \RR^d$.\\
Let $\alpha$ be an eigenvalue of the  expansive (resp. similar)
linear map $M$, and let $v\neq 0$ be an  eigenvector associated to
$\alpha$. We have $\|Mv\|=|\alpha|\|v\|$, which implies that
$|\alpha|\geq \lambda$ (resp. $|\alpha|=\lambda$) and then,
$|\det(M)|\geq \lambda^d$ (resp. $|\det(M)|=\lambda^d$). Thus, if
$\Theta$ is a Borel set in $\RR^d$, we obtain
$$
\vol(M\Theta)=|\det(M)| \vol(\Theta)\geq \lambda^d \vol(\Theta)
\mbox{
 if $M$ is expansive.}
$$
$$
\vol(M\Theta)=|\det(M)| \vol(\Theta)=\lambda^d \vol(\Theta) \mbox{
 if $M$ is a  similarity.}
$$
%%%%%%%%%%%%%%%%%%%%%%%%%%%%%%%%%%%%%%%%%%%%%%%%%%%%%%%%%%%
%\subsubsection{ Self-affine and self-similar tilings. }
%%%%%%%%%%%%%%%%%%%%%%%%%%%%%%%%%%%%%%%%%%%%%%%%%%%%%%%%%%%%%

Let $\A$ be a finite collection of prototiles and let $M$ be a
expansive linear  map on $\RR^d$. A {\em substitution} is a function
$S$ on the set of prototiles $\A$  that associates to each $p$ in
$P$ a patch $S(p)$ such that
\begin{itemize}
\item
the support of $S(p)$ is $M\supp(p)$.
\item
for every $q\in \A$ there exist $n_{p,q}\geq 0$ and $v_{p,q,k}\in
\RR^d$ for each $1\leq k\leq n_{p,q}$, such that
$$
S(p)=\{q+v_{p,q,k}: 1\leq k\leq n_{p,q},\,  q\in \A\}.
$$
\end{itemize}
The {\em substitution matrix} of $S$ is the matrix $A\in
\M_{\A\times\A}(\ZZ^+)$ which contains, in the coordinate $(p,q)$,
the number of different tiles in $S(p)$ which are equivalent to $q$.
That is, $A_{p,q}=n_{p,q}$ for each $p, q\in \A$.\\
The substitution $S$ can be defined  on $T(\A)$ in the following
way: if $t$ is a tile in $\T\in T(\A)$, such that $t$ is equivalent
to the prototile $p\in \A$, we define $$S(t)=S(p)+Mv,$$ where $v\in
\RR^d$ is such that $\supp(t)=\supp(p)+v$. Then, we define
$$S(\T) =\bigcup_{t\in \T}S(t)\in T(\A).$$

The substitution is {\em primitive} if $A$ is primitive, that is, there exists $k>0$ such that
$A^k>0$. In this case, the Perron eigenvalue of $A$ is $|\det(M)|$ (\cite{So1}). 
%Indeed, for $$v=(\vol(\supp(p)))_{p\in \A}\in (\RR^+)^{|\A| }$$ we have 
%$Av=|\det(M)|v$. This means that $|\det(M)|$ is a positive eigenvalue of $A$ with an associated positive eigenvector. 
%From the Perron-Frobenius Theorem \cite{PF}, we conclude that $|\det(M)|$  is the Perron eigenvalue of $A$.

In this paper, we always suppose that   $S$ is primitive. 

The {\em substitution tiling system} associated to $S$ is the tiling system
$(X_S,\RR^d)$, where $X_S$ is the space of all the tilings $\T$ in
$T(\A)$ such that for every patch $P$ of $\T$ there exist a
prototile $p\in \A$ and $k>0$ satisfying $P\subseteq S^k(p)$. The
action of $\RR^d$ on $X_S$ is the translation. Because $S$ is primitive,
there always exist a tiling $\T_0\in T(\A)$ and $k_0>0$ such that
$S^{k_0}(\T_0)=\T_0$. It is classical (in the primitive case) that
$\Omega_{\T_0}=X_S=X_{S^k}$ for every $k>0$. So, without loss of
generality we can suppose that $S(\T_0)=\T_0$. In addition, we will
always  suppose that the fixed point of $S$ satisfies the FPC. In
this case   $X_S$ is a compact metric space and $(X_S,\RR^d)$ is
minimal.

A tiling $\T$ in $T(\A)$ which   satisfies the  FPC  is {\em self-affine} if it is the fixed point of a
substitution. The tiling $\T$ is said to be {\em self-similar} if it
is the fixed point of a substitution $S$ which is defined by a
similarity $M$ with constant $\lambda$ (For more details see \cite{So1}). We say $\lambda$ is the {\em  stretching factor} of $S$ or $\T$. 

Let $\T_0$ be a self-similar tiling which is the fixed point of a primitive
substitution $S$ satisfying the FPC. The following two results are
included in \cite{So}.

\begin{lemma}
\label{LR} $\T_0$ is linearly recurrent.
\end{lemma}

\begin{lemma}
\label{cota} There exists $N>0$ such that if $P$ is a patch in
$\T_0$ whose support contains a ball of radius $R$, then whenever $P+v$ is a patch of $\T_0$ with $v>0$, $\left\| v  \right\| > \frac{R}{N}$.
\end{lemma}

These two lemmata mean that the minimal distance between two
equivalent patches in a self-similar tiling is neither  too large
nor too small compared to their sizes.

%%%%%%%%%%%%%%%%%%%%%%%%%%%%%%%%%%%%%%%%%%%%%%%%%%%%%%%%%%%%%%%%%%%%%%%%
\section{Frequencies}
%%%%%%%%%%%%%%%%%%%%%%%%%%%%%%%%%%%%%%%%%%%%%%%%%%%%%%%%%%%%%%%%%%%%%%%
\label{frequencies}
Consider a tiling $\T$ of $\RR^d$. For a set $F\subseteq \RR^d$, we
write
$$
\T[[F]]=\{t\in \T: t\cap F\neq \emptyset\}.$$ 
A $\T$-{\em corona} is a
patch $\T[[\supp(t)]]$, where $t$ is a tile in $\T$.
Remark that for some $\epsilon \in \RR^d$ we could have 
$\T[[F+\epsilon]] = \T[[F]]$.
To avoid this situation we define, for $v\in \RR^d$, $\T [F,v] = \T [[F]] - v$.
When $F$ is a ball $B_R(v)$ we write $\T [B_R(v)]$ instead of $\T [B_R(v),v]$.

In the sequel we suppose that  $\T_0$ is a self-similar tiling which
is the fixed point of a primitive substitution $S$, with stretching factor $\lambda$, satisfying the FPC.

%%%%%%%%%%%%%%%%%%%%%%%%%%%%%%%%%%%%%%%%%%%%%%%%%%%%%%%%%%%%%%%
\subsection{Van Hove sequences.}
%%%%%%%%%%%%%%%%%%%%%%%%%%%%%%%%%%%%%%%%%%%%%%%%%%%%%%%%%%%%%%%
In order to define the notion of frequency of a patch we need the
concept of Van Hove sequences.

\medskip

Let $P$ be a patch in $\T_0$ and let $\Theta\subset \RR^d$. Denote
by $L_P(\Theta)$ the number of patches included in $\T_0\cap \Theta$
which are equivalent to $P$ (\cite{So1}).

A sequence $(\Theta_n)_{n\geq 0}$ of subsets of $\RR^d$ is a {\em
Van Hove} sequence if for any $r>0$,
$$
\lim_{n\to \infty}\frac{\vol((\partial
\Theta_n)^{+r})}{\vol(\Theta_n)}=0,$$ where
$$
\Theta^{+r}=\{x\in \RR^d: \dist(x,\Theta)\leq r\},$$ and 
$\partial\Theta$ is the border of $\Theta$.

In \cite{So1}, it was shown  for any patch $P$ in $\T_0$ there is a
number $\freq(P)>0$ such that for any Van Hove sequence
$(\Theta_n)_{n\geq 0}$,
$$
\lim_{n\to \infty}\frac{L_P(\Theta_n)}{\vol(\Theta_n)}=\freq(P).
$$

Suppose that $P$ and $Q$ are two patches in $\T_0$. In order to simplify the notation, we will write $L_{P}(Q)$, $\vol(P)$ and 
$(\partial P)^{+r}$ instead of $L_{P}(\supp(Q))$, $\vol(\supp(P))$ and 
$(\partial \supp(P))^{+r}$ respectively.

\medskip

It is easy to show that $(M^n\Theta)_{n\geq 0}$ is a Van Hove sequence when  $M:\RR^d\to \RR^d$ is an expansive linear map and $\Theta$  is a compact subset of $\RR^d$ with non empty interior and such that $\vol(\partial \Theta)=0$.  Consequently, to compute $\freq(P)$ we will use  the following limit
$$
\freq(P)=\lim_{k\to
\infty}\frac{L_P(S^k(p))}{\vol(S^k(p))},
$$
for any prototile $p$ in $\A$.

\subsection{Patch frequencies of a self-similar tiling }
The next proposition extends a result of C. Holton and L. Zamboni \cite{HZ} obtained for minimal substitution subshifts. But before
we will need the following technical lemma:

\begin{lemma}
\label{corona}
Suppose that $\T$  satisfies the FPC. Then there exists a constant $\eta>0$
such that for every $y\in \RR^d$ the ball $B_{\eta}(y)$ is contained in the support of a corona in $\T$.
\end{lemma}

\begin{proof}
 Let $t$ be a tile in $\T$ .  The number
$$\eta_t=\dist(\partial t, \partial\T[[\supp(t)]] )$$ is positive for every  tile $t$.   The FPC implies there is a finite number of coronas up translations. Hence  we get
$$
\eta=\min\{ \eta_t: t\in \T\}>0.
$$
Notice that the set
$$
\{x\in \RR^d: \dist(x,t)\leq \eta\}
$$
is contained in the support of $\T[[\supp(t)]]$ for every tile $t$ in $\T$. Thus if $y$ is a point in $\RR^d$ belonging to the tile $t\in \T$ then the ball  $B_{\eta}(y)$ 
is contained in the support of $\T[[\supp(t)]]$.
\end{proof}

\begin{prop}
\label{frecuencia} 
There exists a finite set $F\subset \RR$ such
that  for every patch $P$   in $\T_0$  satisfying $P=\T_0[B_R(y)]$,
for some $R>0$ and $y\in \RR^d$,
$$
\freq(P)=\frac{f}{\lambda^{dk}},
$$
where $f\in F$ and $k>0$ is such that
$$\lambda^{k-1}\eta\leq \diam(P)< \lambda^k\eta,$$
with $\eta$ is the constant of Lemma \ref{corona}.

\end{prop}
\begin{proof}
Let $\A$ be the prototile set associated to $\T_0$. We define
$$\overline{l}=\max\{\diam(p): p\in \A\}.$$

Let $P$ be a patch in $\T_0$ such that $P=\T_0[[B_R(y)]]$, for some
$R>0$ and $y\in \RR^d$. This implies that
\begin{equation}
\label{diam}
\diam(P)\leq 2(R+\overline{l}).
\end{equation}

Let $k\geq 0$ be such that
\begin{equation}
\label{diam2}
\lambda^{k-1}\eta\leq \diam(P)< \lambda^k\eta.
\end{equation}
By Lemma \ref{corona}, there exists a corona $B$ which support contains the ball $B_{\eta}(M^{-k}y)$. Because the support of $S^k(B)$ contains the ball  $B_{\lambda^k\eta}(y)$, by (\ref{diam2}) we deduce that
$S^k(B)$ contains  the patch $P$. From Lemma \ref{cota}, we have
\begin{equation}
\label{1} L_{P}(S^k(B))\leq
\frac{\vol(S^k(B))}{\vol(B_{\frac{R}{N}}(0))}= 
\frac{\lambda^{kd}}{\frac{R^d}{N^d}}\frac{\vol(B)}{\vol(B_1(0))}.
\end{equation}
From  (\ref{diam}) and (\ref{diam2}) we obtain
$$
\label{2} \frac{1}{2(R+\overline{l})}\leq\frac{1}{\diam(P)}\leq
\frac{1}{\lambda^{k-1}\eta},
$$
which implies there exists $C$ not depending on $k$ such that
\begin{equation}
\label{3} \frac{\lambda^{kd}}{R^d}\leq
\left(\frac{2\lambda}{\eta-\frac{2\overline{l}}{\lambda^{k-1}}}\right)^d\leq
C.
\end{equation}
 From (\ref{1}) and (\ref{3}) we
conclude there exists a constant $K$, independent on $P$, $k$ and
$B$, such that
$$
L_P(S^k(B))\leq K.
$$
Let $P'$ be any patch in $\T_0$ and let $D$ be the set of all the
$\T_0$-coronas, up to translation. We have
$$
L_P(S^k(P'))=\sum_{B\in D}L_B(P')N(P',P,B)$$

where $N(P',P,B)$ is some integer in $\{0,\cdots,
L_P(S^k(B))\}\subseteq \{0,\cdots, K\}$. Thus, for $p\in \A$ and
$n>k$,
\begin{eqnarray*}
\frac{L_P(S^n(p))}{\vol(S^n(p))} & = & \frac{L_P(S^k(S^{n-k}(p)))}{\vol(S^n(p))}\\
                                 & = &  \sum_{B\in D}\frac{L_B(S^{n-k}(p))N(S^{k-n}(p),P,B)}{\vol(S^n(p))}\\
                                 & = & \sum_{B\in D}\frac{L_B(S^{n-k}(p))}{\vol(S^{n-k}(p))}\frac{\vol(S^{n-k}(p))}{\vol(S^n(p))} N(S^{k-n}(p),P,B)\\
                                 & = & \frac{1}{\lambda^{kd}}\sum_{B\in D}\frac{L_B(S^{n-k}(p))}{\vol(S^{n-k}(p))} N(S^{k-n}(p),P,B)
\end{eqnarray*}
%$$
%\frac{L_P(S^n(p))}{\vol(S^n(p))}=\frac{L_P(S^k(S^{n-k}(p)))}{\vol(S^n(p))}=\sum_{B\in
%D}\frac{L_B(S^{n-k}(p))N(S^{k-n}(p),P,B)}{\vol(S^n(p))}
%$$

%$$
%=\sum_{B\in
%D}\frac{L_B(S^{n-k}(p))}{\vol(S^{n-k}(p))}\frac{\vol(S^{n-k}(p))}{\vol(S^n(p))}
%N(S^{k-n}(p),P,B)
%$$
%$$
%=\frac{1}{\lambda^{kd}}\sum_{B\in
%D}\frac{L_B(S^{n-k}(p))}{\vol(S^{n-k}(p))} N(S^{k-n}(p),P,B)
%$$
Because $N(S^{k-n}(p),P, B)$ is in $\{1,\cdots, K\}$ for every $n>k$,
we can take a convergent subsequence to obtain
\begin{eqnarray*}
\freq(P) & = & \frac{1}{\lambda^{kd}}\lim_{n\to \infty}\sum_{B\in D}\frac{L_B(S^{n-k}(p))}{\vol(S^{n-k}(p))}N(S^{k-n}(p),P,B)\\
         & = & \frac{1}{\lambda^{kd}}\sum_{B\in D}\freq(B)N(P,B),
\end{eqnarray*}

%$$
%\freq(P)=\frac{1}{\lambda^{kd}}\lim_{n\to \infty}\sum_{B\in
%D}\frac{L_B(S^{n-k}(p))}{\vol(S^{n-k}(p))}
%N(S^{k-n}(p),P,B)=\frac{1}{\lambda^{kd}}\sum_{B\in D}\freq(B)N(P,B),
%$$
where $N(P,B)$ is some integer in $\{0,\cdots, K\}$ for every $B\in
D$. Because $D$ is finite, to conclude it suffices to take
$$
F=\left\{\sum_{B\in D}\freq(B)N_B: N_B\in \{0,\cdots, K\}\right\}.
$$

%there exists a finite subset $F$ of $\RR$
%such that $$\sum_{B\in D}\freq(B)N(P,B)\in F \mbox{ for every patch
%$P$,}$$ which implies that $\freq(P)=\frac{f}{\lambda^{dk}}$, where
%$f$ is some element in $F$.

\end{proof}

\begin{remark}
\label{unique-ergodicity}{\rm  From \cite{So1} we know $(\Omega_{\T_0},\RR^d)$ is uniquely ergodic. Hence, the frequency of a patch $P$ does not
depend on the tiling. That is, $\freq(P)$ is the same for every $\T$
in $\Omega_{\T_0}$.}
\end{remark}

\subsection{Patch frequency in the factor }
The next result extends Proposition \ref{frecuencia} to tiling factors of self-similar tiling systems.
The main problem we have to overcome is that the factor map is not necessarily given by a sliding block code.
Hence the first part of the next proof consists in selecting carefully the preimages of a given patch $P$ by means of a finite induction procedure.
Then, we show that the frequency of the patch $P$ is the sum of the frequencies of the selected patches.

\begin{prop}
\label{principal} Let $\T$ be a   non periodic
tiling. %satisfying the FPC.  
If there exists a factor map $\pi:\Omega_{\T_0}\to
\Omega_{\T}$ then there exists a finite set $F\subseteq \RR$ such
that for every  patch $P$ in $\T$ satisfying $P=\T[B_R(y)]$, for some 
$R>0$ and $y\in \RR^d$,
$$
\freq(P)=\frac{f}{\lambda^{dk}},
$$
where $f\in F$ and $k>0$ is such that 
$$
\eta\lambda^{k-3}\leq \diam(P) < \eta\lambda^{k-1},
$$
if $R$ is big enough. 

\end{prop}

\begin{proof}
 Let $\T_2\in \Omega_{\T}$ and let $\T_1\in \Omega_{\T_0}$ be such that $\pi(\T_1)=\T_2$. Let $s_0>0$ be the constant of Lemma  \ref{semi-sliding-block-code}.

The linear recurrence of $\T_1$ implies
that  the tiling $\T_2$ is also linearly recurrent.  Let $L$  be the
constant of linear recurrence of $\T_1$ and let $M$ and $N$  be the constants
of Lemma \ref{cota} associated to $\T_1$ and $\T_2$ respectively. We set 
$$K=\max\{(8LN)^d, (8LM)^d \}$$
and
$$\eta_i=\max\{\diam(t): t \mbox{ is a tile in } \T_i\},  \mbox{ for } i\in \{1,2\}.$$ 
Let $\varepsilon>0$. Let $R_{\varepsilon}>0$ be the positive number associated to $\varepsilon$ 
as in Lemma \ref{semi-sliding-block-code}. Notice that  $R_{\varepsilon}$ can be chosen big enough in order that

\begin{equation}
\label{Repsilon}
R_{\varepsilon}\geq \max\left \{ \begin{array}{c}
                                  s_0+\eta_1+\eta_2+\varepsilon\\
                                  4N(2K+1)\varepsilon\\
                                   2M\varepsilon-s_0\\
                                   2(\eta_1+\varepsilon)-(s_0+\eta_2)\\
                                  \eta\lambda^{\lceil \log_{\lambda}\frac{2\eta_1}{\eta(\lambda-1)} \rceil}\\
                                  \eta\lambda^{\lceil \log_{\lambda}\frac{2(s_0+\eta_1+\eta_2+2\varepsilon)}{\eta(\lambda-1)} \rceil+2}\\
                                  \eta/2
                                  \end{array}\right.
\end{equation}
%$$ 
%R_{\varepsilon}   \geq  \max\left \{s_0+\eta_1+\eta_2+\varepsilon;\,  4N(2K+1)\varepsilon;\,  2M\varepsilon-s_0;\,  %2(\eta_1+\varepsilon)-(s_0+\eta_2) \right\} 
%$$
Let $R\geq R_{\varepsilon}$ and let $P=\T_2[B_R(y)]$, $y\in \RR^d$.

Suppose that $v_1,\cdots, v_l$ are all the points in
$B_{2L(R+s_0+\varepsilon+\eta_1+\eta_2)}(0)$ such that
$$\T_2[B_{R}(v_i)]=P.$$ If $v_i\neq v_j$ we have
$\|v_i-v_j\|>\frac{R}{N}$. This implies that in a ball of radius
$\frac{R}{2N}$ there is at most one point $v$ such that
$\T_2[B_{R}(v)]=P$. Using (\ref{Repsilon}) It follows that in
$B_{2L(R+s_0+\varepsilon+\eta_1+\eta_2)}(0)$ there are at most

$$\frac{\vol( B_{2L(R+s_0+\varepsilon+\eta_1+\eta_2)}(0)) }{\vol(B_{ \frac{R}{2N}}(0) ) }\leq (8LN)^d \leq K$$
points $v$ such that  $\T_2[B_{R}(v)]=P$.  %Notice there exists a
%constant $K$, which   depends neither on $\varepsilon$ nor on $R$, such that
%$$ \frac{\vol(B_{2L(R+s_0+\varepsilon+\eta_1+\eta_2)}(0)) }{\vol( B_{ \frac{R}{2N}}(0) ) }
%\leq K.$$
This implies that for any patch $P$ we have $l\leq K$. % and that we can take 
%$\varepsilon$ to be small enough  in order that
%$$
%\varepsilon\leq \min \left\{\frac{R}{2(2K+1)2N}, \frac{R+s_0}{2M}
%\right\},
%$$
%where $M$ is the constant of Lemma \ref{cota} associated to $\T_1$.
%\begin{figure}
%\label{figure1}
%\input{patch1.pstex_t}
%\end{figure}

For every $1\leq i\leq l$ we set $$P_i=\T_1[B_{R+s_0+\eta_2}(v_i)].$$ Now, for every $1\leq i\leq l$ we will
define, by induction on $i$, $k_i$ different patches as follows (see figure \ref{figure2}).

For $i=1$, we take all the patches $P'$ in $\T_1$ satisfying the following two conditions:
\begin{eqnarray}
P'  & = & \T_1[B_{R+s_0+\eta_1+\eta_2+2\varepsilon}(v)]  \mbox{ for some $v\in \RR^d$}\\
P_1 & = &  \T_1[B_{R+s_0+\eta_2}(v)].
\end{eqnarray}
%$$
%P'=\T_1[B_{R+s_0+\eta_1+\eta_2+2\varepsilon}(v)] \mbox{ for some
%$v\in \RR^d$}
%$$
%and
%$$
%P_1=\T_1[B_{R+s_0+\eta_2}(v)].
%$$
Because $\T_1$ satisfies the FPC, there exists a finite number $k_1$
of different patches satisfying the previous condition. We call
these patches $P_{1,1},\cdots, P_{1,k_1}$. Moreover, $k_1$ is
bounded by  $K$.  Indeed, if
$v$ and $v'$ are two different points in $\RR^d$ such that
\begin{eqnarray*}
P_{1,j} & =           & \T_1[B_{R+s_0+\eta_1+\eta_2+2\varepsilon}(v)]\\
P_{1,i} & =           & \T_1[B_{R+s_0+\eta_1+\eta_2+2\varepsilon}(v')],
\end{eqnarray*}
%$$
%P_{1,j}=\T_1[B_{R+s_0+\eta_1+\eta_2+2\varepsilon}(v)]
%$$
%and
%$$
%P_{1,i}=\T_1[B_{R+s_0+\eta_1+\eta_2+2\varepsilon}(v')],
%$$
for some $1\leq i,j\leq k_1$, then
$$P_1=\T_1[B_{R+s_0+\eta_2}(v)]=\T_1[B_{R+s_0+\eta_2}(v')].$$ From
Lemma \ref{cota}, this implies that 
$$\|v-v'\|>\frac{R+s_0+\eta_2}{M}.$$
It follows that in a ball of radius 
$\frac{R+s_0+\eta_2}{2M}$ 
there is at most one point $w$ which is the center of some
$P_{1,j}$. Since $\T_1$ is linearly recurrent with constant $L$ and for every $1\leq j\leq k_1$
$$\diam(P_{1,j})\leq 2(R+s_0+\eta_1+\eta_2+2\varepsilon) +2\eta_1, $$
all the patches $P_{1,j}$ appear in the ball
$B_{2L(R+s_0+2\eta_1+\eta_2+2\varepsilon)}(0)$ in $\T_1$. Using (\ref{Repsilon}) this
implies 
$$
k_1\leq
\frac{\vol(B_{2L(R+s_0+2\eta_1+\eta_2+2\varepsilon)}(0))}{\vol(B_{\frac{R+s_0+\eta_2
}{2M}}(0))}\leq (8LM)^d \leq K.
$$
%different patches $P_{1,j}$. Because there exists a constant $K'$, which  depends neither on $\varepsilon$ nor on $R$ nor on $P_1$   such that
%$$
%\frac{\vol(B_{2L(R+s_0+2\eta_1+\eta_2+2\varepsilon)}(0))}{\vol(B_{\frac{R+s_0+\eta_2
%}{2M}}(0))}\leq K',
%$$
%we conclude that $k_1\leq K'$, for any patch $P$.

%\medskip

For $1<i\leq l$, we take all the patches $P'$ in $\T_1$ satisfying the following three conditions:
\begin{eqnarray}
P'  & = & \T_1[B_{R+s_0+\eta_1+\eta_2+2\varepsilon}(v)] \mbox{ for some $v\in \RR^d$},\\
P_i & = & \T_1[B_{R+s_0+\eta_2}(v)],
\end{eqnarray}
%$$
%P'=\T_1[B_{R+s_0+\eta_1+\eta_2+2\varepsilon}(v)] \mbox{ for some
%$v\in \RR^d$},
%$$
%$$
%P_i=\T_1[B_{R+s_0+\eta_2}(v)],
%$$
\begin{equation}
\label{mayores}
\mbox{ if }  \T_1[B_{R+s_0+\eta_2}(v+t)]=P_j   \mbox{ for some } t\in
B_{2\varepsilon}(0)  \mbox{ then }  j\geq i.
\end{equation}
As for the case $i=1$, we remark there is a finite number $k_i$ of
different
 patches satisfying the previous conditions, and that $k_i$ is smaller than $K$.  We call these patches $P_{i,1},\cdots, P_{i,k_i}$.

\begin{figure}
\input{patch2.pstex_t}
\caption{ }\label{figure2}
\end{figure}

\begin{remark}
\label{remark1}
 {\rm The linear recurrence of $\T_1$ and (\ref{Repsilon}) imply that if $v\in \RR^d$
satisfies 
$$\T_1[B_{R+s_0+\eta_1+\eta_2+2\varepsilon}(v)]=P_{i,j},$$ for some $1\leq i\leq l$ and $1\leq j\leq k_i$, then
$\T_1[B_{R+s_0+\eta_2}(v+t)]\neq P_i$ for every $t\in
B_{2\varepsilon}(0)\setminus\{0\}$.}
\end{remark}

\begin{remark}
\label{remark2}
{\rm From Remark
\ref{remark1}  and from (\ref{mayores}), if $v\in \RR^d$ satisfies 
$$\T_1[B_{R+s_0+\eta_1+\eta_2+2\varepsilon}(v)]=P_{i,j},$$ for some $1\leq i\leq l$
and $1\leq j\leq k_i$,
then $\T_1[B_{R+s_0+\eta_2}(v+t)]\neq P_s$ for every $1\leq s\leq i$
and $t\in B_{2\varepsilon}(0)\setminus \{0\}$.}
\end{remark}

\begin{remark}{\rm  From the construction of the patches $P_{i,j}$, if $v\in
\RR^d$  satisfies 
$$
\T_1[B_{R+s_0+\eta_2}(v)]=P_i,
$$
for some $1\leq i\leq l$ and, $j>i$  whenever  $\T_1[B_{R+s_0+\eta_2}(v+t)]=P_j$ for some  $t\in
B_{2\varepsilon}(0)\setminus\{0\}$, then
$$
\T_1[B_{R+s_0+\eta_1+\eta_2+ 2\varepsilon}(v)]=P_{i,k},
$$
for some $1\leq k\leq k_i$.}
\end{remark}

In the sequel we will show that $\freq(P)=\sum_{i=1}^{l}\sum_{j=1}^{k_i}\freq(P_{i,j})$.

\begin{lemma}
Let $v\in \RR^d$ be such that 
$$
\T_1[B_{R+s_0+\eta_1+\eta_2+2\varepsilon}(v)]=P_{i,j},
$$
for some $1\leq i\leq l$ and $1\leq j\leq k_i$.
Then there exists a point $w(v)  \in B_\epsilon (v)$ verifying 
$\T_2[B_{R}(w(v))]=P$
Moreover, if $v'\not =v$ then $w(v')  \neq w(v)$, and, 
\begin{equation}
\label{freqP1}
\sum_{i=1}^{l}\sum_{j=1}^{k_i}\freq(P_{i,j})\leq \freq(P).
\end{equation}
\end{lemma}

\begin{proof}
Consider
$v\in \RR^d$ such that

 $$\T_1[B_{R+s_0+\eta_1+\eta_2+2\varepsilon}(v)]=P_{i,j},$$
for some $1\leq i\leq l$ and $1\leq j\leq k_i$. Since
$\T_1[B_{R+s_0+\eta_2}(v)]=P_i$, we have
$$
(\T_1+v)\cap B_{R+s_0+\eta_2}(0)=(\T_1+v_i)\cap B_{R+s_0+\eta_2}(0).
$$
Thus from Lemma \ref{semi-sliding-block-code} we obtain that there exists
 $t\in B_{\varepsilon}(0)$ verifying
$$
(\T_2+v+t)\cap B_{R+\eta_2}(0)=(\T_2+v_i)\cap B_{R+\eta_2}(0),
$$
which implies that $\T_2[B_{R}(v+t)]=P$. Now, if $v'\in \RR^d$ is
another point such that
  $$\T_1[B_{R+s_0+\eta_1+\eta_2+2\varepsilon}(v')]=P_{i',j'},$$
  for some $1\leq i'\leq l$ and $1\leq j'\leq k_i'$, in a similar way we get that there exists
  $t'\in B_{\varepsilon}(0)$ satisfying $\T_2[B_{R}(v'+t')]=P$. Suppose that $v+t=v'+t'$. This implies that
  $\|v-v'\|<2\varepsilon$, i.e $v-v'\in B_{2\varepsilon}(0)$.
  But since
\begin{eqnarray*}
P_{i,j} & = & \T_1[B_{R+s_0+\eta_1+\eta_2+2\varepsilon}(v)],\\
P_i     & = & \T_1[B_{R+s_0+\eta_2}(v)], \\
P_{i'}  & = & \T_1[B_{R+s_0+\eta_2}(v+(v'-v))],
\end{eqnarray*}
%$$
%\T_1[B_{R+s_0+\eta_1+\eta_2+2\varepsilon}(v)]=P_{i,j},$$
%$$\T_1[B_{R+s_0+\eta_2}(v)]=P_i \mbox{ and }
%\T_1[B_{R+s_0+\eta_2}(v+(v'-v))]=P_{i'},
%$$
the condition (\ref{mayores}) implies that $i'\geq i$. In the same
way we obtain that $i'\leq i$, which implies $i=i'$. Since
$2\varepsilon<\frac{R+s_0}{M}$, we get that $v'-v=0$. Hence we deduce that
it is possible to associate to each $v$ in $\RR^d$
which satisfies
$$\T_1[B_{R+s_0+\eta_1+\eta_2+2\varepsilon}(v)]=P_{i,j},$$ for some
$1\leq i\leq l$ and $1\leq j\leq k_i$, a point $w(v)\in \RR^d$
verifying $$\T_2[B_{R}(w(v))]=P,$$ and such that $w(v)\neq w(v')$ if
$v\neq v'$. Thus we deduce that
\begin{equation*}
\sum_{i=1}^{l}\sum_{j=1}^{k_i}\freq(P_{i,j})\leq \freq(P).
\end{equation*}
\end{proof}

\begin{lemma}
Let $v\in \RR^d$ be such that
$
\T_2[B_R(v)]=P.
$ 
Then there exists a point $ p(v)  \in B_{(2l+1)\epsilon} (v)$  verifying 
$$
\T_1[B_{R+s_0+\eta_1+\eta_2+2\varepsilon}(p(v))]=P_{i,j},
$$
for some $1\leq i\leq l$ and $1\leq j\leq k_i$.
Moreover, if $v'\not =v$ then $p(v')  \neq p(v)$, and, 
\begin{equation}
\label{freqP2}
\sum_{i=1}^{l}\sum_{j=1}^{k_i}\freq(P_{i,j})\geq \freq(P).
\end{equation}
\end{lemma}

\begin{proof}
Let $v\in \RR^d$ be such
that $$\T_2[B_R(v)]=P,$$  and consider
$$P'=\T_1[B_{R+s_0+\eta_2+\varepsilon}(v)].$$ Since $L$ is the
constant of linear recurrence of $\T_1$ and
$$
\diam(P')\leq 2(R+s_0+\eta_2+\varepsilon)+2\eta_1,
$$
there exists a translated  of $P'$ which support is included in the ball
$$B_{2L(R+s_0+\eta_1+\eta_2+\varepsilon)}(0).$$ In other words,
there exists $v'\in B_{2L(R+s_0+\eta_1+\eta_2+\varepsilon)}(0)$
such that the support of the patch $\T_1[[B_{R+s_0+\eta_2+\varepsilon}(v')]]$ is
contained in the ball
$B_{2L(R+s_0+\eta_1+\eta_2+\varepsilon)}(0)$   and  satisfies
\begin{eqnarray*}
P' & = & \T_1[B_{R+s_0+\eta_2+\varepsilon}(v')]\\
   & = & \T_1[B_{R+s_0+\eta_2+\varepsilon}(v)].
\end{eqnarray*}
This implies that
 $$(\T_1+v)\cap B_{R+s_0+\eta_2}(0)=(\T_1+v')\cap B_{R+s_0+\eta_2}(0).$$ So, from Lemma
 \ref{semi-sliding-block-code} there exists $t\in B_{\varepsilon}(0)$ verifying
$$
(\T_2+v'+t)\cap B_{R+\eta_2}(0)=(\T_2+v)\cap B_{R+\eta_2}(0).
$$
It follows that $\T_2[B_{R}(v'+t)]=P$ and, since
 $v'+t$ is in $B_{L(R+s_0+\eta_1+\eta_2+\varepsilon)}(0)$, we deduce that
 $v'+t=v_i$, for some $1\leq i\leq l$. Because
$\T_1[B_{R+s_0+\eta_2}(v'+t)]= P_i$ is included in
$\T_1[B_{R+s_0+\eta_2+\varepsilon}(v')]=P'$, we obtain that
$$\T_1[B_{R+s_0+\eta_2}(v+t)]=P_i.$$ 

\begin{figure}
\input{patch3.pstex_t}
\caption{ }\label{figure3}
\end{figure}

Now, we will  show that in the ball
$B_{(2l+1)\varepsilon}(v)$ there is a point $p(v)$ such that
$$\T_1[B_{R+s_0+\eta_1+\eta_2+2\varepsilon}(p(v))]=P_{m,j},$$ for
some $1\leq m\leq l$ and $1\leq j\leq k_m$. For that, consider the
following algorithm (see figure \ref{figure4}):

\begin{enumerate}
\item[Step $0$:]
We put $v_0=v+t$ and $i_0= i$.
\item[Step $1$:]
We have $\T_1[B_{R+s_0+\eta_2}(v_0)]=P_{i_0}$.\\
If $\T_1[B_{R+s_0+\eta_2}(v_0+s)]=P_j$ for some $s\in
B_{2\varepsilon}(0)$ implies $j\geq i_0$, then from the
definition of the patches $P_{i,k}$  we obtain that
$$\T_1[B_{R+s_0+\eta_1+\eta_2+2\varepsilon}(v_0)]=P_{i_0,m},$$ for some $m$ in $\{1,\cdots, k_{i_0}\}$.
\item[Step $2$:] If there exists $s\in B_{2\varepsilon}(0)$ such that
$\T_1[B_{R+s_0+\eta_2}(v_0+s)]=P_j$ with $j<i_0$, then we put
$$i_0=\min\{j: \exists    s\in B_{2\varepsilon}(0) \mbox{ such
that } \T_1[B_{R+s_0+\eta_2}(v_0+s)]=P_j\}.$$ If   $s\in
B_{2\varepsilon}(0)$ is  such that
$\T_1[B_{R+s_0+\eta_2}(v_0+s)]=P_{i_0}$ then we put $v_0=v_0+s$.
With these new values of $v_0$ and $i_0$ we go to the step $1$.

\end{enumerate}
This algorithm finishes in at most $l$ steps. The result is a point
$p(v)=v_0$ which distance to $v$ is  at most
$(2l+1)\varepsilon$ and such that $$\T_1[B_{R+s_0+\eta_1+\eta_2+2\varepsilon}(v_0)]=P_{i_0,m},$$  for some $m$ in $\{1,\cdots, k_{i_0}\}$.

\begin{figure}
\input{patch4.pstex_t}
\caption{}\label{figure4}
\end{figure}

If $w\in \RR^d$ is another point satisfying $\T_2[B_R(w)]= P$, we have
\begin{eqnarray*}
\frac{R}{N} & \leq & \|v-w\|\\ 
            & \leq & \|p(v)-v\|+\|p(v)-p(w)\|+\|p(w)-w\|\\
            & \leq &  2(2l+1)\varepsilon + \|p(v)-p(w)\|.
\end{eqnarray*}
 Thus we get 
$$0<\frac{R}{2N}<\frac{R}{N}-2(2l+1)\varepsilon\leq
\|p(v)-p(w)\|.$$  This implies it is possible to
associate to each $v$ in $\RR^d$  which satisfies $
\T_2[B_{R}(v)]=P$ a point $p(v)\in \RR^d$ verifying
$$\T_1[B_{R+s_0+\eta_1+\eta_2+2\varepsilon}(p(v))]=P_{i,j},$$ for
some $1\leq i\leq l$ and $1\leq j\leq k_i$, and such that $p(v)\neq
p(w)$ if $v\neq w$. Hence we deduce that
\begin{equation*}
\freq(P)\leq \sum_{i=1}^{l}\sum_{j=1}^{k_i}\freq(P_{i,j}).
\end{equation*}
\end{proof}

From (\ref{freqP1}) and (\ref{freqP2}) we get
\begin{equation}
\label{freqP3}
\freq(P)= \sum_{i=1}^{l}\sum_{j=1}^{k_i}\freq(P_{i,j}).
\end{equation}

As $R>\eta/2$, there exists $k>0$ such that
\begin{equation}
\label{nn}
\eta\lambda^{k-2}\leq 2(R+s_0+\eta_1+\eta_2+2\varepsilon) < \eta\lambda^{k-1}.
\end{equation}
Since $$2(R+s_0+\eta_1+\eta_2+2\varepsilon)\leq \diam(P_{i,j})
\leq 2(R+s_0+\eta_1+\eta_2+2\varepsilon)+2\eta_1$$
and $R\geq \eta\lambda^{\lceil \log_{\lambda}\frac{2\eta_1}{\eta(\lambda-1)} \rceil}$,  we have
$$
\eta\lambda^{k-2}\leq \diam(P_{i,j})<\eta\lambda^k.
$$
Hence, by Proposition  \ref{frecuencia}, we get
%if we take a  big enough $R>0$, we can ensure, by Proposition \ref{frecuencia}, that
%there exists $k>0$ such that for every $1\leq i\leq l$ and $1\leq j
%\leq k_i$
$$
\freq(P_{i,j})\in \left\{ \frac{f}{\lambda^{dk}},
\frac{f}{\lambda^{d(k-1)}}: f\in F\right\},
$$
where $F$ is the finite set of Proposition \ref{frecuencia}. Thus we obtain 
$$
\freq(P)=\frac{f}{\lambda^{dk}},
$$
where $f$ is an element in
$$
F'=\left\{\sum_{i=1}^Kf_i\ : f_i\in F\cup \lambda^dF, \, \forall \, 1\leq
i\leq K \right\},
$$
which is a finite subset of $\RR^d$.  

Notice that 
$$
2R\leq \diam(P) \leq 2(R+\eta_2).
$$
Thus from (\ref{nn}) we have
$$
\eta\lambda^{k-2}-2(s_0+\eta_1+\eta_2+2\varepsilon)\leq \diam(P)<\eta\lambda^{k-1},
$$
and by the choice of $R$ in (\ref{Repsilon}), we obtain
$$
\eta\lambda^{k-3}\leq \diam(P) < \eta\lambda^{k-1}.
$$
\end{proof}

\section{Proof of Theorem \ref{mainresult}}

From Proposition \ref{principal}, there exist two finite sets $F_1$ and
$F_2$ such that for  $R>0$ and $P=\T[B_R(0)]$ there exist $k_1$ and $k_2$ such that
$$
\freq(P)=\frac{f_1}{\lambda_1^{k_1}}=\frac{f_2}{\lambda_2^{k_2}},
$$
for some $f_1\in F_1$ and $f_2\in F_2$.

Because $F_1$ and $F_2$ are finite, we can find $a\in F_1$, $b\in
F_2$, $n_2> n_1$, $m_2>m_1$ and  patches $P_1$ and $P_2$ in $\T$
such that
$$
\freq(P_1)=\frac{a}{\lambda_1^{n_1}}=\frac{b}{\lambda_2^{m_1}},
$$
$$
\freq(P_2)=\frac{a}{\lambda_1^{n_2}}=\frac{b}{\lambda_2^{m_2}}.
$$
This implies that
$$
\lambda_1^{n_2-n_1}=\lambda_2^{m_2-m_1},
$$
which means that $\lambda_1$ and $\lambda_2$ are multiplicatively
dependent.

 \vspace{5mm}

{\bf Acknowledgment: } The authors acknowledge financial support
from  Nucleus Millennium PF04-069-F, Ecos-Conicyt program C03E03, Fondecyt de Iniciaci\'on 11060002 and thank 
the Laboratoire Ami\'enois de Math\'ematiques Fondamentales et
Appliqu\'ees of the Universit\'e de Picardie Jules Verne, Centro de Modelamiento Matem\'atico and the Departamento de Ingenier\'{\i}a
Matem\'atica of the Universidad de Chile where part
of this work has been done.

\end{document}